\documentclass[reqno,a4paper]{amsart}
 \usepackage{a4wide}
 \newtheorem{thm}{Theorem}[section]
 \newtheorem{cor}[thm]{Corollary}
 \newtheorem{lem}[thm]{Lemma}
 
 \theoremstyle{definition}
 
 \theoremstyle{remark}

 \numberwithin{equation}{section}
 
 \newcommand{\F}{\mathcal{F}}
 \newcommand{\Id}{\mathsf{Id}}
 \newcommand{\A}{\mathcal{A}}
 \newcommand{\M}{\mathcal{M}}
 \newcommand{\norm}[1]{\left\|#1\right\|}
 \newcommand{\R}{\mathbb{R}}
 \usepackage{nicefrac}
 \usepackage{color}
 \renewcommand{\d}{\mathsf{d}}
 \newcommand{\abs}[1]{\left|#1\right|}

\begin{document}

%-------------------------------------------------------------------------
% editorial commands: to be inserted by the editorial office
%
%\firstpage{1} \volume{228} \Copyrightyear{2004} \DOI{003-0001}
%\seriesextraline{This is the Concrete Title of this Book\br H.E. R and S.T.C. W, Eds.}
%\firstpage{1}

\title[Residual Growth Control for General Maps]
 {Residual Growth Control for General Maps and an Approximate Inverse Function Result}

%----------Author 1
\author[Amrein]{Mario Amrein}

\address{%
Gertrudstrasse 8\\
CH 8400 Winterthur\\
Switzerland}

\email{mario.amrein@zhaw.ch}

%\thanks{The author thanks the anonymous referee for the helpful comments that improved the quality of the manuscript.}
\subjclass{58Z99; 34G99; 34A34; 37N30}

\keywords{Continuous Newton methods, Global Newton methods, Residual growth control, Inverse function results,  Path following scheme's, Zero finding homotopy.}

\date{June 27, 2024}
%%% ----------------------------------------------------------------------

\begin{abstract}
The need to control the residual of a potentially nonlinear function
$\F$ arises in several situations in mathematics. For example, computing the zeros of a given map, or the reduction of some cost function during an optimization process are such situations. In this note, we discuss the existence of a curve $t\mapsto x(t)$ in the domain of the nonlinear map $\F$ leading from some initial value $x_0$ to a value $u$ such that we are able to control the residual $\F(x(t))$ based on the value $\F(x_0)$. 
More precisely, we slightly extend an existing result from J.W. Neuberger by proving the existence of such a curve, assuming that the directional derivative of $\F$ can be represented by $x \mapsto \A(x)\F(x_0)$, where $\A$ is a suitable defined operator. The presented approach covers, in case of $\A(x) = -\Id$, some well known results from the theory of so-called \emph{continuous Newton methods}. Moreover, based on the presented results, we discover an approximate inverse function result. 
\end{abstract}

%%% ----------------------------------------------------------------------
\maketitle
%%% ----------------------------------------------------------------------

\section{Introduction}
\label{sec:intro}
The object of this work is a (possibly) nonlinear map $\F:X\rightarrow \mathbb{R}^{m}$ with $X\subset \R^n$. Based on some initial guess $x_{0}\in X$ and the curve $t\mapsto x(t)\in X$ with $x|_{t=0}=x_0$, we want to control the residual
\begin{equation}
\label{eq:01}
t\mapsto \F(x(t)),
\end{equation}
using the initial residual $\F(x_0)$.
In many situations, we are interested in zeros of $\F$. Apart of finding zeros for $\F$, e.g., in machine learning tasks, one often seeks to minimize some cost function $\F$ using an initial condition $x_0$, i.e., for some $t_1>0$, we would like to have 
\[
\norm{\F(x(t_1))}\leq c \cdot \norm{\F(x_{0})}
\]
for a constant $c<1$.
However, in case of $X$ being a subset of $\mathbb{R}^{m}$ with initial value $x_0 \in X$, and assuming that the derivative of $\F$ is invertible, the well-known ODE, also called Davidenko's method \cite{Davidenko}, given by
\begin{equation}
\label{eq:02}
\begin{cases}
\dot{x}&=-[\F'(x)]^{-1}\F(x_0)\\
x|_{t=0}&=x_0
\end{cases}
\end{equation}
may serve as a possible scheme to find a zero for $\F$ starting from the initial value $x_0\in X$. Indeed, under suitable 
assumptions on $\F$, there exists a solution $t\mapsto x(t)$ on the interval $[0,1]$ solving \eqref{eq:02} and leading to a zero of $\F$. Notice that this solution is implicitly given by
\[
\F(x(t))=(1-t)\F(x_0),
\]
i.e., $\F(x(1))=0$. Let us briefly outline the connection between the well-known continuous Newton method given by the initial value problem
\begin{equation}
\label{eq:02a}
\begin{cases}
\dot{z}&=-[\F'(z)]^{-1}\F(z)\\
z|_{t=0}&=x_0
\end{cases}
\end{equation}
and the initial value problem from \eqref{eq:02}. First of all, notice that a solution $t\mapsto z(t)$ of \eqref{eq:02a} satisfies $\F(z(t))=\F(x_0)\mathrm{e}^{-t}$. Thus, given a solution $t\mapsto x(t)$ of \eqref{eq:02}, we obtain the solution $z(t)$ based on
\[z(t)=x(w(t)), \qquad \text{with} \qquad  t\mapsto w(t) = 1-\mathrm{e}^{-t}.\]
On the other hand, given a solution of \eqref{eq:02a}, we obtain a solution $x(t) $ of \eqref{eq:02} using
\[
x(t)=z(w^{-1}(t)), \qquad \text{with}   \qquad t  \mapsto w^{-1}(t)=\ln\left(\frac{1}{1-t}\right).
\]
Regarding the initial value problem \eqref{eq:02}, a somewhat natural approach is to replace the ODE \eqref{eq:02} by the more general ODE
\begin{equation}
\label{eq:03}
\begin{cases}
\dot{x}&=\M(x)\F(x_0)\\
x|_{t=0}&=x_0,
\end{cases}
\end{equation}
where $X \ni x \mapsto \M(x)\in \mathbb{R}^{n\times m}$. In case of $n = m$ and assuming that $\M(x)$ is approximately $ -[\F'(x)]^{-1}$ , we still may expect that the residual $\F$ decreases on a trajectory $t\mapsto x(t)$ solving  \eqref{eq:03}.

Let us take a closer look at this situation.
Assume for the moment that on the set $X\subset \mathbb{R}^{n}$ the map $X\ni x\mapsto \F'(x)\in \mathbb{R}^{m\times n}$ exists. In addition, let $t\mapsto x(t)$ be a solution of \eqref{eq:03}. Formally we now use the derivative of $\F$, multiply \eqref{eq:03} from the left with $\F'(x)$, and obtain: 
\begin{equation}
\label{eq:04}
\frac{\d }{\d t}\F(x(t))=\A(x(t))\F(x_0), \quad \A(x):=\F'(x)\M(x) \in \mathbb{R}^{m\times m}.
\end{equation}
Integrating \eqref{eq:04} on $[0,T]$, we end up with
\begin{equation}
\label{eq:05}
\F(x(t))=\F(x_0)+\int_{0}^{t}{\A(x(s))\F(x_0)\d s}, \quad t \in [0,T].
\end{equation}
We can rewrite \eqref{eq:05} as follows:
\begin{equation}
\label{eq:06}
\F(x(t))=\F(x_0)(1-t)+\int_{0}^{t}{(\A(x(s))+\Id_{\R^{m}})\F(x_0)\d s}, \quad t \in [0,T],
\end{equation}
where $\Id_{\R^{m}}$ denotes the identity $\mathbb{R}^{m} \ni x \mapsto \Id_{\R^{m}}(x)=x$.
Henceforth, if $\A(x)+\Id_{\R^{m}}$ is bounded on a set $U\subset \mathbb{R}^{m}$ containing the trajectory of $t\mapsto x(t)$, say 
\begin{equation}
\label{eq:07}
\kappa :=\sup_{x\in U}{\norm{\A(x)+\Id_{\R^{m}}}}=\sup_{x\in U}{\left(\sup_{w \in \mathbb{R}^{m},\norm{w}\neq 0}{\frac{\norm{(\A(x)+\Id_{\R^{m}})(w)}}{\norm{w}}}\right)},
\end{equation}
we get
\begin{equation}
\label{eq:resid_bound}
\norm{\F(x(1))}\leq \kappa \norm{\F(x_0)}.
\end{equation}
Here $\norm{\cdot}$ stands for the standard Euclidean norm, i.e., $\norm{x}=\left(\sum_{i=1}^{n}{x_{i}^2}\right)^{\nicefrac{1}{2}}$ for $ x\in \mathbb{R}^{n}$. As a conclusion, we see that we are able to control the residual $\F(x)$ by the two components $\kappa $ and $\norm{\F(x_0)}$.
Remember that in case of $\M(x)=-[\F'(x)]^{-1}$ in \eqref{eq:03} and \eqref{eq:04}, i.e., for problem \eqref{eq:02}, we readily infer \[\A(x)=-\Id_{\mathbb{R}^{m}},\] i.e., we get  $\kappa = 0$ and therefore $x(1)$  is a zero for $\F$.
Based on the results given in \cite{Neuberger2,Neuberger}, we see in this work that the residual bound from \eqref{eq:resid_bound} can be obtained using only some notion of directional derivative for $\F$:
\begin{equation}
\label{eq:directional}
\F'(x)h:=\lim_{t\searrow 0}{\frac{\F(x+th)-\F(x)}{t}}, \quad x,h \in \mathbb{R}^{n}.
\end{equation}  
In \cite{Neuberger2}, the author shows the existence of a zero $u$ for $\F$ under the assumption that $x$ and $h$ are taken from some suitable subsets of $\mathbb{R}^n$ such that the directional derivative of $\F$ from \eqref{eq:directional} exists and satisfies 
\[
\F'(x)h=-\F(x_0).
\]
We see that the proof given in \cite{Neuberger2} can be extended in such a way that we obtain the residual bound from \eqref{eq:resid_bound}, assuming that the directional derivative satisfies
\[
\F'(x)h=\A(x)\F(x_0).
\]
Therefore, the result from \cite{Neuberger2} can be seen as a special case of the residual bound \eqref{eq:resid_bound} for $\kappa = 0$. Based on the same assumption on the directional derivative as in the work \cite{Neuberger2}, where the author proves the existence of a zero $u$ for $\F$, in \cite{Neuberger}, the forgoing result has been extended by showing the existence of a continuous function $x:[0,1]\to \mathbb{R}^{n}$ with
\[
\F(x(t))=(1-t)\F(x_0).
\]
As shown in \cite{Neuberger2} and \cite{Neuberger}, using some suitable notion of compactness, these results are also valid in the infinite dimensional case. In this work, we see that the existence of a function $t\mapsto x(t)$ can be extended such that there holds
\begin{equation}
\label{eq:general_path}
\norm{\F(x(t))-(1-t)\F(x_0)}\leq  \kappa t \norm{\F(x_0)},
\end{equation}
i.e., the result given in \cite{Neuberger} is simply the case with $\kappa = 0$ in \eqref{eq:general_path}.

\subsubsection*{Outline}
In Section \ref{sec:Main_Results}, we first show the existence of an element $u\in \mathbb{R}^{n}$ satisfying the residual bound \eqref{eq:resid_bound}. In addition, we discuss some implications and applications based on this residual bound. In a second part, we address the existence of a path $t\mapsto x(t)$ that satisfies
\eqref{eq:general_path}. 
We again discuss this result and some implications thereof.
We finally discuss these results in the infinite dimensional case and give some concluding remarks in Section \ref{sec:concl}.
The presented results rely on ideas from \cite{Neuberger2,Neuberger}.

\section{Main results}
\label{sec:Main_Results}

For $r>0$, we denote by $b_{r}(x)$ the open, and by $B_{r}(x)$ the closed ball in $ X\subset{\mathbb{R}^{n}}$ with $x\in X$. We start with the following result:

\begin{thm}
\label{theo:02}
Given $r>0, x_0\in \mathbb{R}^{n}$ and let $\F:B_{r}(x_0)  \to \mathbb{R}^{m}$ 
be a continuous function on the ball $B_{r}(x_0)$. In addition, assume that the map $B_{r}(x_0)\ni x \mapsto \A(x) \in \mathbb{R}^{m\times m} $ is continuous.
If for all $x \in b_{r}(x_0)$, there exists a direction $h\in B_{r}(0)$ such that
\begin{equation}
\label{eq:directional_derivative}
\lim_{t\searrow 0}{\frac{\F(x+th)-\F(x)}{t}}=\A(x)\F(x_0),
\end{equation}
then there exists an element $u\in B_{r}(x_0)$ with
\[
\norm{\F(u)}\leq \kappa \norm{\F(x_0)},
\]
where the constant $\kappa $ is given by
\[
\kappa = \max_{x\in B_{r}(x_0)}{\norm{\A(x)+\Id_{\R^{m}}}}.
\]
\end{thm}

\begin{proof}
First of all we note that by continuity of $x\mapsto \norm{\A(x)+\Id_{\R^{m}}}$ the constant $\kappa $ is finite.
For the proof, we follow the ideas of \cite{Neuberger} and define for given $ \varepsilon >0$ the set: 
\[
S_{\varepsilon}=\{s\in [0,1]|\exists y \in B_{rs}(x_0) \ \text{s.t.} \ \norm{\F(y)-(1-s)\F(x_0)}\leq \varepsilon s + \kappa s \norm{\F(x_0)} \}.
\] 
Note that $S_\varepsilon$ is a closed set. Indeed, for a given sequence $ (s_{k}) \subset S_{\varepsilon}$ with $s_k \to s$, there is a sequence $ (y_k) $ with $y_k\in B_{r s_k}(x_0)$ and 
\begin{equation}
\label{eq:continuity}
\norm{\F(y_k)-(1-s_k)\F(x_0)}\leq \varepsilon s_k + \kappa s_k \norm{\F(x_0)}.
\end{equation}
Note that $ (y_{k})\subset B_r(x_0)$. Thus by compactness of $B_{r}(x_0)$, we can choose a subsequence $ (y_{k_{j}})$ with $y_{k_{j}}\to y \in B_{r}(x_0)$. In addition we have
\[
\norm{y-x_0}\leq \norm{y-y_{k_{j}}}+\norm{y_{k_j}-x_{0}}\leq \norm{y-y_{k_{j}}}+rs_{k_{j}} \to rs,
\]
i.e., $y\in B_{rs}(x_0)$. Together with \eqref{eq:continuity}, using the continuity of $\F$, we conclude that $s_{k}\to s $ in $S_{\varepsilon}$, i.e., $S_{\varepsilon}$ is a closed subset of [0, 1], and therefore $\sup{S_{\varepsilon}}\in S_\varepsilon$. We see that $\sup{S_{\varepsilon}}=1$. 
Let us assume by contradiction that $\sup{S_{\varepsilon}}= m <1$. 
For $m \in S_\varepsilon$ we can choose $y\in B_{rm}(x_0)$ with
\begin{equation}
\label{eq:122}
\norm{\F(y)-(1-m)\F(x_0)}\leq \varepsilon m+\kappa m \norm{\F(x_0)}.
\end{equation}
Since $y\in B_{rm}(x_0)\subset b_{r}(x_0)$, there exists $h \in B_{r}(0)$ such that assumption \eqref{eq:directional_derivative} holds. Therefore we can choose $ \delta \in (0,1-m]$ so small such that 
\begin{equation}
\label{eq:11}
\norm{\F(y+\delta h)-\F(y)-\delta \A(y)\F(x_0)}\leq  \varepsilon \delta .
\end{equation}
Using \eqref{eq:122} and \eqref{eq:11} we observe:
\[
\begin{aligned}
& \norm{\F(y+\delta h)-(1-(\delta +m))\F(x_0)} \\
\leq & \norm{\F(y+\delta h)-\F(y)-\delta \A(y)\F(x_0)}\\
+ &  \norm{\F(y)-(1-m)\F(x_0)}+\delta\norm{(\A(y)+\Id_{\mathbb{R}^{m}})\F(x_0)}\\
\leq & \varepsilon \delta + \varepsilon m+\kappa m \norm{\F(x_0)}+\delta \kappa \norm{\F(x_0)}\\
= & \varepsilon( \delta+m)+ \kappa (m+\delta) \norm{\F(x_0)}.
\end{aligned}
\]
Since $\norm{(y+\delta h)-x_0}\leq \norm{y-x_0}+\delta \norm{h}\leq r(m + \delta )$, we see that 
$ (y+ \delta h)\in B_{r(\delta + m)}(x)$ and 
we deduce $\delta + m \in S_{\varepsilon}$ contradicting the assumption $\sup{S_{\varepsilon}}=m$. Hence $\sup{S_{\varepsilon}}=1$ and therefore for all $\varepsilon>0$, we find $y_{\varepsilon}\in B_{r}(x_0)$ with
\[
\norm{\F(y_{\varepsilon})}\leq \varepsilon + \kappa  \norm{\F(x_0)}.
\]
Again by compactness of $B_{r}(x_0)$, we obtain the existence of $u\in B_{r}(x_0)$ with
\[
\norm{\F(u)}\leq \kappa \norm{\F(x_0)}.
\]
\end{proof}
We now consider some immediate consequences of Theorem \ref{theo:02}:
\begin{cor}
\label{cor:02}
Suppose the assumptions of Theorem \ref{theo:02} hold for any $x=x_0\in\mathbb{R}^{n}$.
Then, for any $i\in \mathbb{N}$, there exists $u_i\in\mathbb{R}^{n}$ with
\[
\norm{\F(u_{i})}\leq \kappa^{i}\norm{\F(u_0)}.
\]
\end{cor}

\begin{proof}
This follows by induction. For $i=1 $ this holds true by Theorem \ref{theo:02} with $x_0=u_0$. Thus, assume that $\norm{\F(u_{i-1})}\leq \kappa^{i-1}\norm{\F(u_0)}$. We set $ x_0=u_{i-1}$ in Theorem \ref{theo:02} and obtain $u_i$ with $\norm{\F(u_i)}\leq \kappa \norm{\F(u_{i-1})}$ from where we deduce the desired result. 
\end{proof}

\begin{cor}
Suppose the assumptions of Theorem \ref{theo:02} hold for any $x=x_0\in\mathbb{R}^{n}$. In addition, suppose that the sequence $(u_i)_{i \in \mathbb{N}}$ from Corollary \ref{cor:02} remains bounded and that $\kappa <1$. Then there exists at least one zero $u$ for $\F$.
\end{cor}

\begin{proof}
This simply follows from Corollary \ref{cor:02} together with the boundedness of the sequence $(u_i)_{i \in \mathbb{N}}$. 
\end{proof}

Next, we consider an approximate inverse function result:

\begin{thm}
\label{theo:invers_finite}
Given $r>0$, $B_{r}(0)\subset \mathbb{R}^{n} $, a continuous map $\Psi:\mathbb{R}^{n}\to \mathbb{R}^{m}$ with $\Psi(0)=0$, and the continuous matrix-valued map $\A:B_{r}(0) \rightarrow \mathbb{R}^{m\times m}$, let $g\in \mathbb{R}^{m} $ be such that for all $x\in b_{r}(0)$, there is $h\in B_{r}(0)$ with 
\[
\lim_{t\searrow 0}{\frac{\Psi(x+th)-\Psi(x)}{t}}=\A(x)g.
\]
Then, there exists $u \in B_{r}(0)$ with
\[
\norm{\Psi(u)-g}\leq \kappa \norm{g},
\]
where the constant $\kappa $ is now given by
\[
\kappa = \max_{x\in B_{r}(0)}{\norm{\A(x)+\Id_{\R^{m}}}}.
\]
\end{thm}

\begin{proof}
Set $\F(x):=\Psi(x)-g$ and note that for any $y\in b_{r}(0)$ there is $h \in B_{r}(0)$ such that
\[
\lim_{t\searrow 0}{\frac{\Psi(y+th)-\Psi(y)}{t}}=\lim_{t\searrow 0}{\frac{\F(y+th)-\F(y)}{t}}=\A(y)g. 
\]
Invoking Theorem \ref{theo:02}, we obtain $u\in B_{r}(0)$ with 
\[
\norm{\Psi(u)-g}=\norm{\F(u)}\leq \kappa \norm{\F(0)} = \kappa \norm{g}.
\]
\end{proof}
Let us now discuss the existence of a continuous curve $x:[0,T]\to \mathbb{R}^{n}$ satisfying the bound \eqref{eq:general_path}. 
\begin{thm}
\label{theo:03}
Under the assumptions given in Theorem \ref{theo:02}, there exists a continuous function $x:[0,1]\rightarrow \mathbb{R}^{n}$ such that
\[
\norm{\F(x(t))-(1-t)\F(x_0)}\leq \kappa t \norm{\F(x_0)}.
\]
In addition, $x([0,1])\subset B_{r}(x_0)$.
\end{thm} 
Note that from this result, we immediately obtain Theorem \ref{theo:02} if we set $t=1$. Again, we follow the ideas given in \cite{Neuberger}. In doing so, we use the following result:
\begin{lem}
\label{lem:01}
Given $t \in (0,1)$, $y\in b_{tr}(x_0)$, $s \in (0,1-t]$ and suppose the assumptions from Theorem \ref{theo:02} hold. Then for all $\varepsilon>0$, there exists $w\in B_{rs}(y)$ such that
\begin{equation}
\label{eq:09}
\norm{\F(w)-\F(y)+s\F(x_0)}\leq s\varepsilon + s\kappa \norm{\F(x_0)}.
\end{equation}
\end{lem}

\begin{proof}
Let $t\in (0,1)$. Choose $(y,h) \in b_{tr}(x_0)\times B_{r}(0)$ such that 
\[
\lim_{s\searrow 0}{\frac{\F(y+sh)-\F(y)}{s}}=\A(y)\F(x_0).
\]
Next, we fix a positive number $\varepsilon >0 $ and set $\delta >0 $ such that whenever $s\in (0,\delta]$, 
\[
\norm{\F(y+sh)-\F(y)-s\A(y)\F(x_0)}\leq \varepsilon s.
\]
Again, we define the set
\[
\begin{aligned}
S_{y,\varepsilon}=&\{\delta \in (0,1-t]|\forall s \in (0,\delta] \ \exists w\in B_{sr}(y): \\
&\norm{\F(w)-\F(y)+s\F(x_0)}\leq s\varepsilon + s\kappa \norm{\F(x_0)}\}.
\end{aligned}
\]
As before,  we see that $\sup{S_{y,\varepsilon}} \in S_{y,\varepsilon}$ and equals $1-t$. Let therefore $\delta_{k}>0$ such that $\delta_k\nearrow m = \sup{S_{y,\varepsilon}}$. For each $\delta_k$, we can find $w_k \in B_{r\delta_k}(y)\subset B_{r m}(y)$ such that
\[
\norm{\F(w_k)-\F(y)+\delta_k \F(x_0)}\leq \varepsilon \delta_k+\delta_k \kappa \norm{\F(x_0)}.
\]
By compactness of $B_{r m }(y)$, we can find $w \in B_{r m }(y)$ such that $w_{k_{j}}\to w$ as $j\to \infty$, and the continuity of $\F$ implies 
\begin{equation}
\label{eq:16}
\norm{\F(w)-\F(y)+ m\F(x_0)}\leq m \varepsilon + m \kappa \norm{\F(x_0)},
\end{equation}
i.e., $m\in S_{y,\varepsilon}$.
It remains to show that $m = 1-t$. By contradiction, we assume that $0<m <1-t$ which  is equivalent to $0<1-(t+m)$. 
Note that 
\[
\norm{w-x_0}\leq \norm{w-y}+\norm{y-x_0}\leq r(m + t)<r.
\]
Thus, by assumption, we can find $\widetilde{h} $ such that 
\[
\lim_{s\searrow 0}{\frac{\F(w+s\widetilde{h})-\F(w)}{s}}=\A(w)\F(x_0).
\]
Choose again a positive number $0<d<1-(t + m)$ such that for $0<s\leq d$,
\begin{equation}
\label{eq:17}
\norm{\F(w+s\widetilde{h})-\F(w)-s\A(w)\F(x_0)}\leq \varepsilon s.
\end{equation}
We set $\widetilde{w}=w+s\widetilde{h}$. Invoking \eqref{eq:16} and \eqref{eq:17}, we arrive at
\[
\begin{aligned}
&\norm{\F(\widetilde{w})-\F(y)+(s+m)\F(x_0)}\\
\leq &\norm{\F(\widetilde{w})-\F(w)-s\A(w)\F(x_0)}\\
&+\norm{\F(w)-\F(y)+m \F(x_0)}+s\norm{(\A(w)+\Id)\F(x_0)}\\
\leq & \varepsilon s + m \varepsilon +m \kappa \norm{\F(x_0)}  +s\kappa \norm{\F(x_0)}\\
=& (s+m)\varepsilon+(s+m) \kappa \norm{\F(x_0)}.
\end{aligned}
\]
Moreover, since $\norm{\tilde{w}-y}\leq sr+\norm{w-y}$, we have
$\widetilde{w} \in B_{{r}(m+s)}(y)$, i.e., we get the contradiction $ \sup{S_{y,\varepsilon}} = m < m + d = \sup{S_{y,\varepsilon}}  $.
\end{proof}

\begin{proof}[Proof of Theorem \ref{theo:03}]
We use the ideas from Neuberger \cite{Neuberger}. Indeed, for $k\in \mathbb{N}$, let $p_{k}$ be a partition of the interval $[0,1]$ such that $p_{k+1}$ is a refinement of $p_{k}$. The $n_k + 1$ nodes of $p_k$ will be denoted by
\[
0 = t_{k,0}<t_{k,1}<t_{k,2}<\ldots <t_{k,n_{k}}=1.
\]
E.g., we can choose the sequence of partitions 
\[
p_{k}=\left\{\frac{i}{2^{k}}\right\}_{\{i=0,\ldots,2^{k}\}}, \quad k \in \{1,2,3,\ldots\}, \quad \text{with} \quad t_{k,i}=\frac{i}{2^{k}}.
\]
Next we define the family of functions $x_{k}:[0,1]\rightarrow \mathbb{R}^{n}$ as the piecewise linear interpolant with nodal values
\[
x_{k}(t_{k,i})=w_{k,i} \quad \text{for} \quad i \in\{0,1,\ldots,n_{k}\}.
\]
Here, $x_{k}(t_{k,0})=w_{k,0}$.
For $w_{k,0}=x_{0}$ and $t_{1}>0$, we choose $w_{k,1}\in b_{t_1r}(x_0)$, such that there holds
\[
\begin{aligned}
\norm{\F(w_{k,1})-\F(x_0)+t_{1}\F(x_0)}&\leq \norm{\F(w_{k,1})-\F(x_0)-t_1\A(x_0)\F(x_0)}\\
& \quad +t_1\norm{\A(x_0)\F(x_0)+\F(x_0)}\\
& \leq  \frac{t_{1}}{k}+t_1\kappa \norm{\F(x_0)}.
\end{aligned}
\]  
Since $w_{k,1}\in b_{t_{1}r}(x_0)$, by Lemma \ref{lem:01}, there exists $w_{k,2}\in B_{s_2r}(w_{k,1})$ with $s_2=t_2-t_1\in (0,1-t_1]$ such that
\[
\norm{\F(w_{k,2})-\F(w_{k,1}) + s_{k,2}\F(x_0)}\leq \frac{s_{k,2}}{k}
+\kappa s_{k,2} \norm{\F(x_0)}.
\]
Notice that
\[
\norm{w_{k,2}-x_0}\leq\norm{w_{k,2}-w_{k,1}}+\norm{w_{k,1}-x_0}<(s_2+t_1)r=t_2r,
\]
i.e., $w_{k,2}\in b_{t_2r}(x_0)$.
Thus, again by Lemma \ref{lem:01}, we obtain $w_{k,3}\in B_{s_3r}(w_{k,2})$ with $s_3=t_3-t_2$ such that
\[
\norm{\F(w_{k,3})-\F(w_{k,2})+s_{k,3}\F(x_0)}\leq \frac{s_{k,3}}{k}
+\kappa s_{k,3} \norm{\F(x_0)}.
\]
Hence, for $ i\in \{1,2,3,\ldots,n_{k}\}$,  we inductively get
\[
\norm{\F(x_{k}(t_{k,i}))-\F(x_{k}(t_{k,i-1}))+s_{k,i}\F(x_0)}\leq \frac{s_{k,i}}{k}
+\kappa  s_{k,i} \norm{\F(x_0)}.
\]
Summing this inequality from $i=1$ up to $ i= j_k\leq n_{k}$ and using the triangle inequality, we obtain

\begin{equation}
\label{eq:12}
\norm{\F(x_{k}(t_{k,j_k}))-\F(x_0)+t_{k,j_k}\F(x_0)}\leq \frac{t_{k,j_k}}{k}
+ \kappa \norm{\F(x_0)}t_{k,j_k}.
\end{equation}
Remember that $x_{k}(t_{k,0})=x_0$. 
We need to check that
\begin{enumerate}
\item[i)]
for each $ t\in [0,1]$, the sequence $\{x_{k}(t)\}_{k\in \mathbb{N}}$ remains within $B_{r}(x_{0})$ and
\item[ii)]
in addition, for all $k$ and $ s,t\in [0,1]$, the functions are Lipschitz with
\[
\norm{x_{k}(t)-x_{k}(s)}\leq r\abs{t-s}.
\]
\end{enumerate}
We first address $\mathrm{i})$: For $t\in [t_{k,0},t_{k,1}]$ we observe:
\[
\norm{x_{k}(t)-x_0}=\frac{\norm{w_{k,1}-x_0}}{\abs{t_{k,1}-t_{k,0}}}\abs{t-t_{k,0}}\leq rs_{k,1} < r.
\]
Let $t\in [t_{k,1},t_{k,2}]$:
\[
\norm{x_{k}(t)-x_0}\leq \norm{w_{k,2}-w_{k,1}}+\norm{w_{k,1}-x_0}
\leq r(s_{k,2}+s_{k,1})
\leq rt_{k,2}
<r.
\]
Clearly we can repeat this process until we reach the interval $ [t_{k,n_k-1},t_{k,n_k}]$.
Let us now consider $\mathrm{ii})$: 
Choose $t,s\in [0,1]$ with $0\leq s <t$ and assume that for $j \in \mathbb{N}$, we have 
$s,t \in [t_{k,j}, t_{k,j+1}]$. We observe:
\[
\norm{x_{k}(t)-x_{k}(s)}=\frac{\norm{w_{k,j}-w_{k,j-1}}}{t_{k,j}-t_{k,j-1}}\abs{t-s}\leq \frac{rs_{k,j}}{t_{k,j}-t_{k,j-1}}\abs{t-s}=r\abs{t-s}.
\]
%Thus, since $z_{k}$ is lipschitz on each interval $[t_{k,j},t_{k,j+1}]$, $z_{k}$ is uniformly Lipschitz on all of $[0,1)$ with lipschitz constant $r$.
Next assume $0\leq s<t$ with $t_{k,j}\leq t < t_{k,j+1}, t_{k,l}\leq s < t_{k,l+1}$ for $ l<j $. We observe:
\[
\begin{aligned}
&\norm{x_k(t)-x_k(s)}\\
\leq &\norm{x_k(t)-x_k(t_{k,j})}+\norm{x_k(t_{k,j})-x_k(t_{k,j-1})}+\ldots \\
&\ldots + \norm{x_k(t_{k,l+2})-x_k(t_{k,l+1})}+\norm{x_k(t_{k,l+1})-x_k(s)}\\
\leq & r\abs{t-t_{k,j}}+r\abs{t_{k,j}-t_{k,j-1}}+\ldots r\abs{t_{k,l+2}-t_{k,l+1}}+r\abs{t_{k,l+1}-s}\\
\leq & r\abs{t-s}.
\end{aligned}
\]
Hence, by Arzela-Ascoli we can choose a sequence $\{k_{p}\}_{p=1,2,3,\ldots} \subset \mathbb{N}$ such that 
\[
x_{k_p}\to x \in C([0,1];\mathbb{R}^{n}) \quad \text{for} \quad p\to \infty.
\]
Let $t\in [0,1]$. Then, for all $k\in\mathbb{N}$, there exists $j_{k}$ such that 
\[
t_{k,j_{k}} \leq t  \leq t_{k,j_{k}+1}.
\]
Thus $\lim_{k\to \infty}{t_{k,j_k}}=t$
and by continuity of $\F$, we have 
\[
\lim_{p\to \infty}{\F(x_{k_p}(t_{k_{p},j_{k_p}}))}=\F(x(t))
\]
with
\[
\norm{\F(x(t))-(1-t)\F(x_0)}\leq \kappa t \norm{\F(x_0)}.
\]
\end{proof}
\begin{cor}
Under the assumptions of Theorem \ref{theo:03}, there exists $u \in B_{r}(x_0)$ and a constant $\kappa\geq 0$ such that
\[
\norm{\F(u)}\leq \kappa \norm{\F(x_0)}.
\]
\end{cor}

\begin{proof}
Set $u:=\lim_{t\nearrow 1}{z(t)}$ and apply Theorem \ref{theo:03}.
\end{proof}

\begin{cor}
Let $\F$ and $\A$ be defined and continuous on all of $\mathbb{R}^{n}$. Suppose that for each $x\in \mathbb{R}^{n}$, there exists $r>0$ such that $\forall y \in b_{r}(x) $, there is $h \in B_{r}(0)$ with
\[
\lim_{t\searrow 0}{\frac{\F(y+th)-\F(y)}{t}}=\A(y)\F(x).
\]
In addition, assume that 
\[
\kappa=\max_{x\in \mathbb{R}^{n}}{\norm{\A(x)+\Id_{\mathbb{R}^{m}}}}<\infty.
\]
Then there exists a path $x:[0,\infty)\rightarrow \R^{n}$ with $x(0)=x_0$ such that for any $n\in \mathbb{N}_{0}$, there holds 
\[
\norm{\F(x(n))}\leq \kappa^{n} \norm{\F(x_0)}.
\]
Moreover, if the sequence $(x(n))$ remains bounded and $\kappa<1$, there exists a zero for $\F$. 
\end{cor}

\begin{proof}
Let $i\geq 0$ and $x(s):=x_{i}(s-i)$ for $ i\leq s \leq i+1$ where $x_i$ denotes the path from Theorem \ref{theo:02} with
\[
\norm{\F(x_{i}(t))-(1-t)\F(x_{i})}\leq \kappa t \norm{\F(x_{i})}, \quad x(i)=x_{i}=x_{i}(0)=x_{i-1}(1) 
\]
from where we obtain inductively
\[
\begin{aligned}
i=0: &\quad \norm{\F(x_{0}(t))-(1-t)\F(x_{0})}\leq \kappa t \norm{\F(x_{0})}, \quad x_{0}=x_{0}(0),\\
i=1: &\quad \norm{\F(x_{1}(t))-(1-t)\F(x_{1})}\leq \kappa t \norm{\F(x_{1})}, \quad x_{1}=x_{1}(0) = x_{0}(1),\\
i=2: &\quad \norm{\F(x_{2}(t))-(1-t)\F(x_{2})}\leq \kappa t \norm{\F(x_{2})}, \quad x_{2}=x_{2}(0)=x_{1}(1),\\
\vdots & \hspace{7.6cm} \vdots \\
i=j: &\quad \norm{\F(x_{j}(t))-(1-t)\F(x_{j})}\leq \kappa t \norm{\F(x_{j})}, \quad x_{j}=x_{j}(0)=x_{j-1}(1).\\
\end{aligned}
\]
Thus we obtain for $n = i + 1$:
\[
\begin{aligned}
\norm{\F(x(n))}=\norm{\F(x_{n-1}(1))}& \leq \kappa \norm{\F(x_{n-1}(0))}=\kappa \norm{\F(x_{n-2}(1))}\\
& \leq \kappa^2\norm{\F(x_{n-2}(0))}=\kappa^2\norm{\F(x_{n-3}(1))}\\
& \leq \kappa^3\norm{\F(x_{n-3}(0))}=\kappa^3\norm{\F(x_{n-4}(1))}\\
& \ \ \vdots \hspace{2.9cm} \vdots \\
& \leq \kappa^n\norm{\F(x_{0}(0))}\\
& = \kappa^n\norm{\F(x_{0})}.
\end{aligned}
\]
Assuming that $(x(n))$ is bounded, we obtain a zero $u$ for $\F $ as a limit of a subsequence .
\end{proof}

\subsection{The infinite dimensional case}
%\label{sec:inifinite}
Given the Banach spaces $X,\widetilde{X},Y$ with norms $\norm{\cdot}_{X},\norm{\cdot}_{\widetilde{X}}$ and $\norm{\cdot}_{Y}$. Let $\F $ be a map with domain of definition $\mathbb{D}_{\F}$ in $X$ and taking values in $Y$. For a given $x_0\in X$ and $r>0$, assume that the ball $B_{r}(x_0)$ is a subset of $\mathbb{D}_{\F}$. Moreover let $\widetilde{X}$ be a subset of $X$ compactly embedded in $X$, i.e., any bounded sequence $(x_n)$ in $X$ has a convergent subsequence $(x_{n_{k}})$ in $\widetilde{X}$ converging to an element in $X$. In addition, assume that the restriction of $\F$ onto $\widetilde{X}$ is continuous.
By $B_{r}(x)$ and $b_r(x)$ we denote the closed and open balls in $X$ with center $x$ and radius $r$.
The space of linear and continuous operators between $X$ and $Y$ will be denoted by $\mathcal{L}(X;Y)$. Let $\A$ be a map with domain of definition $\mathbb{D}_{\A}$ in $X$ and taking values in $\mathcal{L}(X;Y)$. Assume that $B_{r}(x_0)$ is a subset of $\mathbb{D}_{\A}$ and that the restriction of $ \A$ onto $\widetilde{X}$ is continuous as a map from $\widetilde{X} $ to  $\mathcal{L}(X;Y)$. Notice that for $L \in \mathcal{L}(X;Y)$:
\[
\norm{L}_{\mathcal{L}(X;Y)}=\sup_{\norm{x}_{X}\neq 0}{\frac{\norm{L(x)}_{Y}}{\norm{x}_{X}}}
\]
We finally assume that
\[
\kappa = \sup_{x\in B_{r}(x_0)}{\norm{\A(x)+\Id_{Y}}_{\mathcal{L}(X;Y)}}<\infty.
\]
The following result is the analogue of Theorem \ref{theo:03} for the infinite dimensional case. Therefore, we obtain this result by following the proof for the finite dimensional case apart from the compactness argument, where we now use the assumption that $\widetilde{X}$ is compactly embedded in $X$.
\begin{thm}
\label{theo:04}
If for all $x \in b_{r}(x_0)$, there exist a direction $h\in B_{r}(0)$ such that
\begin{equation}
\label{eq:directional_derivative_infinite}
\lim_{t\searrow 0}{\frac{\F(x+th)-\F(x)}{t}}=\A(x)\F(x_0),
\end{equation}
then there exists a function $t\mapsto x(t)\in \widetilde{X}$ (continuous as a function into $X$) with $t\in [0,1]$ such that
\[
\norm{\F(x(t))-(1-t)\F(x_0)}_{Y}\leq \kappa t \norm{\F(x_0)}_{Y}
\]
In addition, $x([0,1])\subset B_{r}(x_0)$.
\end{thm} 
We can also adapt Theorem \ref{theo:invers_finite} to the infinite dimesional case:
\begin{thm}
\label{theo:inverse_infinite}
Given $r>0$, $B_{r}(0)\subset X $, a map $\Psi:B_{r}(0)\to Y$ with $\Psi(0)=0$, and the operator valued map $\A:B_{r}(0) \rightarrow \mathcal{L}(X,Y)$, let $g\in Y $ such that for all $x\in b_{r}(0)$, there is $h\in B_{r}(0)$ with 
\[
\lim_{t\searrow 0}{\frac{\Psi(x+th)-\Psi(x)}{t}}=\A(x)g.
\]
Then, if $\A$ and $\Psi$ are continuous on $\widetilde{X}$, there exists $u \in B_{r}(0)$ with
\[
\norm{\Psi(u)-g}\leq \kappa \norm{g},
\]
where the constant $\kappa $ is supposed to be finite with:
\[
\kappa = \sup_{x\in B_{r}(0)}{\norm{\A(x)+\Id_{Y}}_{\mathcal{L}(X;Y)}}.
\]
\end{thm}

\section{Conclusions}
\label{sec:concl}
In this work, we discussed the control of the residual for a given possibly nonlinear map $\F$ starting from the quantity $\norm{\F(x_{0})}$. More precisely: Using some notion of directional derivative, we have seen that we can easily extend the results given in \cite{Neuberger2,Neuberger} to establish the bound:
\[
\norm{\F(u)}\leq \kappa \norm{\F(x_0)}
\] 
where $\kappa $ is an upper bound for the quantity $\norm{\A(x)+\Id}$ and $\A$ satisfies $\F'(x)h= \A(x)\F(x_0)$. Based on this observation, we further extended the path existing result from \cite{Neuberger2,Neuberger} in the case that given $x$, the system $\F'(x)h= \A(x)\F(x_0) $ has a solution $h$, to get the bound:
\[
\norm{\F(x(t))-(1-t)\F(x_0)}\leq \kappa t\norm{\F(x_0)}.
\]
As a result, we therefore are able to control the norm of the homotopy
\[
H(x,t)\equiv \F(x)-(1-t)\F(x_0)
\]
along the derived path $t\mapsto x(t)$ with $H(x_0,0)=0$. Based on this result, we further proved an approximate version of an inverse function type Theorem \ref{theo:inverse_infinite}.
Possible future investigations based on the results of this work are the study of approximate equation solving procedures in the finite dimensional as well as within the context of infinite dimensional problems.

% ------------------------------------------------------------------------

% ------------------------------------------------------------------------
\end{document}